\documentclass[12pt, a4paper]{article}
\usepackage{amssymb,amsthm,amsmath,amscd}     
\usepackage{latexsym}                         
\usepackage{enumerate}

\usepackage{verbatim}

\usepackage{color}
\usepackage{url}

\usepackage[all]{xy}

\usepackage{tikz}
\usetikzlibrary{arrows,automata,backgrounds,calendar}

\newtheorem{thm}{Theorem}[section]
\newtheorem{cor}[thm]{Corollary}
\newtheorem{lem}[thm]{Lemma}

\theoremstyle{definition}
\newtheorem{defn}{Definition}[section]

\theoremstyle{remark}
\newtheorem{rem}{Remark}[section]

\newcommand{\N}{\ensuremath{\mathbb{N}}}

\newcommand{\U}{\ensuremath{\mathbb{U}}}
\newcommand{\Uinf}{\ensuremath{\mathbb{U}_{0,\infty}}}
\newcommand{\Ustrich}{\ensuremath{\mathbb{U}_0'}}
\newcommand{\Z}{\ensuremath{\mathbb{Z}}}

\newcommand{\bra}[1]{\ensuremath{[{#1}]}}
\newcommand{\brainf}[1]{\ensuremath{[{#1}]_\infty}}

\newcommand{\bx}{\ensuremath{[ x ]}}
\newcommand{\bTx}{\ensuremath{[ Tx ]}}

\newcommand{\Tstar}{\ensuremath{T^*}}

\allowdisplaybreaks[1]

\begin{document}

\title{On the $3x+1$ conjecture. }

\author{Peter Hellekalek\thanks{The author is supported by the Austrian Science Fund (FWF):
Project F5504-N26, which is a part of the Special Research Program 
"Quasi-Monte Carlo Methods: Theory and Applications}}

\date{\small Dedicated to Harry (Hillel) Furstenberg}
\date{\today}

\maketitle

\renewcommand{\thefootnote}{}
\renewcommand{\thefootnote}{\arabic{footnote}}
\setcounter{footnote}{0}

\begin{abstract}
In this paper, we discuss the well known $3x+1$ conjecture 
in form of the accelerated Collatz function $T$
defined on the
positive odd integers.
We present a sequence of quotient spaces and further, an invertible map,
which are intrinsically related to the behavior of $T$.
This approach allows to express the $3x+1$ conjecture in  form of 
equivalent problems, which might be more accessible than the original conjecture.
\end{abstract}

\section{Introduction}\label{intro}
 Let $\N_0$ stand for the nonnegative integers, $\N_0=\{0, 1, \ldots \}$,
 let $\U = 2\N_0+1$ be the set of odd positive integers,
 and let
 \begin{equation}\label{def:CollatzFunction}
 T: \U \rightarrow \U,\nonumber \quad
    Tx = (3x+1) 2^{-\nu_2(3x+1)},
 \end{equation}
 where $\nu_2(y)$ denotes the exponent of the largest power of $2$ that divides the integer $y$.
 The map  $T$ is called the reduced or accelerated Collatz function in the literature
 (see \cite{Lagarias10a}).
 Hence, $T17=13$, $T13=5$, and $T5=1$.

 The $3x+1$ conjecture, also known as the Collatz conjecture, 
 states that, starting from any $x\in \U$,
 by iterating $T$ we will eventually end up in the number $1$.
 In other words, for every $x\in \U$,
 there exists $k=k(x)\in \N_0$ such that $T^k x=1$.
Here, $T^0$ stands for the identity map,
and, for $k\in \N$,
$T^k$ is defined recursively by $T^k=T\circ T^{k-1}$.
To give an example, $T^3 17=1$.
We refer the reader to the comprehensive monograph \cite{Lagarias10a}
for details on the $3x+1$ conjecture and for the numerous aspects that have been
studied in this context.

In this paper, 
we show how to associate with $T$ an invertible map $T^*$ on a certain quotient space 
that consists of equivalence classes of odd integers. 
The properties of $T^*$ reflect the behavior of $T$, see Theorem \ref{thm:propertiesOfT}.
Further,
we exhibit several statements that are equivalent to the 3x+1 conjecture,
see Corollary \ref{cor:equivalentresults}.
In the appendix, we present additional concepts to describe the action of  $T$
on $\U_0$.

Our approach for the accelerated Collatz map $T$
may be of interest for any such many-to-one map with a unique fixed point.

\section{Results}\label{s:preliminaries}

For the sake of better readability,
we will write $Tx$ instead of $T(x)$ for the image of $x$ under $T$.
The same slight abuse of notation will apply to the functions $S$  and $f$ below.
All other functions will be written as usual.

\begin{rem}
The outline of our approach is the following:
\begin{enumerate}
\item First, we introduce a map $S: \U\rightarrow \U$ that allows to describe the inverse image
$T^{-1}\{y\}$ of a point $y\in \U$, $y\not\equiv 0  \pmod{3}$, completely.

\item The next idea is to restrict $T$ to the subset $\U_0$ of $\U$. 
$T$ is surjective on $\U_0$, hence we may
introduce the \emph{inverse map} $\tau$, inverse in the sense  $T\circ \tau$ being the identity map on $\U_0$.

\item We then study a sequence of equivalence relations $``\sim_n$'', $n\ge 0$, 
which yields, for every $x\in \U_0$, an increasing sequence of equivalence classes $(\bx_n)_{n\ge 0}$ and
an associated decreasing sequence of positive integers $(\delta_n(x))_{n\ge 0}$.

\item The next idea is to study an equivalence relation $``\sim_\infty$'' on $\U_0$,
which leads to a partition of $\U_0$ into equivalence classes $\bx_\infty$, $x\in \U_0$,
and to minimal elements $\delta_\infty(x)$ with the property $\delta_\infty(x)=\lim_{n\to \infty} \delta_n(x)$.
Further, a bijective map $T^*$ that \emph{mimics} the behavior of $T$ may be defined on the
quotient space $\U_0/\sim_\infty$.

\item The $3x+1$ conjecture is then equivalent to $\U_0 = \bra{1}_\infty$,
and also equivalent to $\delta_\infty(x) = 1$, for all $x\in \U_0$.

\item In the appendix, 
we analyze a closely related equivalence relation on $\U_0$, which results in a partition of $\U_0$ into
\emph{$T$-invariant} subsets.

\item Also in the appendix, we extend the map $T$ to an invertible map $f$ on $\U_0$.
In addition, we provide some concepts for ``bookkeeping'' concerning the classes
$\bra{f^k x}_n$ and the positive integers $\delta_n(f^k x)$, 
where $k\in \Z$ and $n\ge 0$.
\end{enumerate}
\end{rem}

\begin{rem}
Our abstract approach sheds some light on the behavior of the map $T$.
Why do we fail to prove the $3x+1$ conjecture or, at least, some partial results?
What is missing in our study  are \emph{quantitative} results,
for example
\begin{enumerate}
\item A description of the growth behavior of the classes $\bra{1}_n$, in dependence of $n$.
(This might yield a result on the \emph{density} of $\bra{1}_\infty$ in $\U_0$.)
\item Number-theoretical arguments proving --for an appropriate notion of distance-- that the distance between
the class $\bra{1}_\infty$ and each class $\bx_\infty$ can be made arbitrarily small.
Equivalently, one could try to show that the assumption $\delta_\infty(x)> 1$ leads to a contradiction.
(This would yield $\bra{1}_\infty = \bx_\infty$ for all $x\in \U_0$, 
thereby proving the $3x+1$ conjecture.)
\item (Number-theoretic) Arguments showing that there are no periodic points $\bx_\infty$ of $T^*$ with a period larger 
or equal to 2.
(This would imply that there is no periodic point of $T$ with period larger or equal to 2,
which is yet unknown.)
\end{enumerate}
\end{rem}

After this outline of concepts and shortcomings in our approach,
let us look at the details.

\begin{rem}\label{rem:rem1}
The following properties of the map $T$ are well known and elementary to prove.

\begin{enumerate}
\item From the definition of $T$,
we derive the equivalence
\begin{equation}\label{eqn:basic}
y=Tx \ \Leftrightarrow\  3x+1 = y 2^{\nu_2(3x+1)}.
\end{equation}

\item For every $y\in \U$ with $y\equiv 0 \pmod{3}$,
\[
T^{-1}\{y\} = \emptyset.
\]
This follows  from  (\ref{eqn:basic}) for the simple
reason that $3x+1\equiv 1 \pmod{3}$,
whereas $y2^{\nu_2(3x+1)}\equiv 0 \pmod{3}$.

\item The element $1$ is the unique fixed point of $T$, i.e.
\[
\{x\in \U: Tx=x\} = \{1 \}.  
\]
Again, this is a direct consequence of  (\ref{eqn:basic}).
\end{enumerate}
\end{rem}

The following map allows to describe the behaviour of $T$.

\begin{defn}
We define $ S:  \U\rightarrow \U$ as
$Sx = 4x+1$.
 \end{defn}
 
The map $S$ permutes the residue classes modulo 3:
if $x\equiv a \pmod{3}$, then $Sx\equiv a+1 \pmod{3}$.
This simple property will prove to be essential for defining an inverse map associated with $T$,
see Definition \ref{defn:Tinverse}.

The next lemma is part of the `folklore' in the $3x+1$ community.
We present a simple proof, for  the sake of completeness.

 \begin{lem}\label{lem:B}
 For all $x\in \U$, we have
 \begin{equation}
 Tx =  T ( S x).
 \end{equation}
 \end{lem}

 \begin{proof}
 We have
 \[
 T (S x) = T(4x+1) =  (3x+1) 2^{2-\nu_2(12x+4)}.
 \]
Trivially, $\nu_2(12x+4) = 2+\nu_2(3x+1)$.
 \end{proof}

\begin{cor}\label{cor:Sj}
Lemma \ref{lem:B} implies for all $x\in \U$ that $x$ and its iterates
$S^k x$ are mapped to $Tx$:
\[
\forall x\in \U, \forall k\ge 0: \quad Tx= T( S^k x).
\]
In other words, $T= T\circ S^k$ on $\U$, for all $k\ge 0$.
\end{cor}

\begin{rem}\label{rem:rem2}
By induction for $k$ we see that
\[
\forall x\in \U, \forall k\ge 0: \quad  S^k x = 4^k x + (4^k-1)/3.
\]
\end{rem}

\begin{defn}\label{defn:inverseofy}
For  $y\in \U$ with $y\not\equiv 0 \pmod{3}$,
let $\xi(y)$ denote the smallest element of $\U$ that is mapped to $y$ by $T$:
\[
\xi(y) = \min \{x\in \U: Tx=y\}.
\]
\end{defn}

\begin{lem}\label{lem:inverseofy}
Let $y\in \U$ with $y\not\equiv 0 \pmod{3}$.
Then $\xi(y)$ is given as follows.
\begin{enumerate}
\item \label{lem:inverseofy:1} If  $y\equiv 1 \pmod{3}$, then
$\xi(y) = (4y-1)/3$.
\item \label{lem:inverseofy:2} If $y\equiv 2 \pmod{3}$, then
$\xi(y) = (2y-1)/3$.
\end{enumerate}
\end{lem}

\begin{proof}
Suppose that $y\equiv 1 \pmod{3}$. 
Then  equivalence  (\ref{eqn:basic})
implies that $\nu_2(3x+1)$ has to be even.
The smallest solution in $\U$ to (\ref{eqn:basic}) is
the number $x$ with the property $\nu_2(3x+1)=2$.
This yields $\xi(y)=(4y-1)/3$.

The case $y\equiv 2 \pmod{3}$ is treated in the same manner.
\end{proof}

\begin{cor}
Suppose that $y\equiv 0 \pmod{3}$.
Then, for $z\in \{Sy, S^2 y \}$,
the preimage $T^{-1}\{ z\}$ is non-void.
\end{cor}

\begin{lem}
The set $T^{-1}\{Tx\}$ of those elements $z$ of $\U$ that are mapped to $Tx$
is given by  $\xi(Tx)$ and its iterates under $S$:
	\[
	\{ z\in \U: Tz=Tx\} = \{S^k \xi(Tx): \  k\ge 0  \}.
	\]
\end{lem}

\begin{proof}
Suppose first that $y=Tx \equiv 1 \pmod{3}$, and assume that $Tz=Tx$.
It follows from (\ref{eqn:basic}) that $\nu_2(3z+1)\in \{2,4,6, \ldots \}$.
If $\nu_2(3z+1)=2$, then from  Lemma \ref{lem:inverseofy}, Part 1, 
it follows that $z=\xi(Tx)$.
If $\nu_2(3z+1)=4$, then 
\[
3z+1 = 2^4 y = 2^2 (3\xi(Tx)+1).
\]
This implies $z=S\xi(Tx)$.

In the general case, if $\nu_2(3z+1) = 2 + 2k$, with $k\ge 1$,
we have
\[
3z+1 = y 2^{2+2k} = (3\xi(Tx)+1) 4^k.
\]
It follows that
\[
z = 4^k \xi(Tx) + (4^k-1)/3,
\]
from which we derive by Remark \ref{rem:rem2} that $z=S^k \xi(Tx)$.
\end{proof}

\begin{lem}
The set
$\U_0 = \{x\in \U: x\not\equiv 0 \pmod{3}  \}$
has the properties $T\U=\U_0$ and $T\U_0 = \U_0$.
In particular, the map $T:\U_0 \rightarrow \U_0$ is surjective.
\end{lem}

\begin{proof}
By Remark \ref{rem:rem1}(2), we have $T\U \subseteq \U_0$,
hence $T\U_0 \subseteq \U_0$.
If $y\in \U_0$, then by Lemma \ref{lem:inverseofy} there exists $x\in \U$ such that $Tx=y$.
Due to Lemma \ref{lem:B}, 
we may assume $x\in \U_0$.
\end{proof}

\begin{cor}
In order to prove the $3x+1$ conjecture,
it suffices to restrict the map $T$ to the set $\U_0$.
\end{cor}

Hence, from now on, we will study the 3x+1 conjecture for the surjective map
$T: \U_0 \rightarrow \U_0$.
We note that the surjectivity of $T$ implies for all subsets $B$ of $\U_0$,
\begin{equation}\label{eqn:invertingT}
T(T^{-1} B ) = B.
\end{equation}

We employ the well-ordering principle to define some sort of  inverse map associated with $T$.

\begin{defn}\label{defn:Tinverse}
For $x\in \U_0$, define the (quasi-)inverse function $\tau$ of $T$ as follows:
\[
\tau(x) = \min \{ z\in \U_0: Tz=x \}.
\]
\end{defn}

The reader should note that,
for $\xi(x)\in \U_0$,
$\tau(x)= \xi(x)$,
whereas for $\xi(x)\equiv 0 \pmod{3}$, we have $\tau(x)= S\xi(x)$.
To give an example, $\tau(5)=13$, whereas $\xi(5)=3$.
Further, $T\circ \tau$ is the identity map on $\U_0$,
whereas, in general, $\tau(Tx)\neq x$.

The next idea is to generate a series of equivalence relations and, 
hence, a series of quotient spaces and of partitions of $\U_0$.

\begin{defn}
For $x,y\in \U_0$ and $n\in \N_0$, we define the relation ``$\sim_n$'' on $\U_0$ as
\[
x\sim_n y \Leftrightarrow T^n x = T^n y.
\]
Further, we put $\bx_n = \left\{ z\in \U_0: T^n z= T^n x  \right\}$, and
 $\delta_n(x) = \min \bx_n$.
\end{defn}

For all $x\in \U_0$, and all $n\ge 0$,
we have $x\in \bx_n$,
hence $\bx_n\neq \emptyset$ and $\delta_n(x)\le x$.
The set $\bra{x}_0$ consists of the single point $x$.

\begin{lem}\label{lemma:A}
For all $x\in \U_0$ and for all $n\in \N_0$,
the following holds.
\begin{enumerate}
\item\label{lemma:A:1} The relation `$\sim_n$' is an equivalence relation on $\U_0$
	and the set $\bx_n$ is the equivalence class of $x$ 
	with respect to this equivalence relation.
	Further, $\bra{x}_n = \bra{\delta_n(x)}_n$.
\item\label{lemma:A:3} We have strict inclusion $\bx_n \subset \bx_{n+1}$.
\item\label{lemma:A:4} For all $n\ge 1$ and all $k\in \N_0$ such that $S^k x \in \U_0$,
	\[
	\bx_n = \bra{S^k x}_n.
	\]
\item\label{lemma:A:5} We have
	\[
	1\le \delta_n(x) \le \delta_{n-1}(x) \le \dots \le \delta_1(x) \le \delta_0(x) = x.
	\]
\item\label{lemma:A:6} For all $x\in \U_0$ and all $n\ge 0$,
\[
T^{-1} \bra{Tx}_n = \bra{x}_{n+1}.
\]
\end{enumerate}
\end{lem}

\begin{proof}
Ad \ref{lemma:A:1}. This is easy to verify.

Ad  \ref{lemma:A:3}.
The inclusion $\bx_n \subseteq \bx_{n+1}$ is trivial.
It follows from the definition of these two sets.
In order to prove strict inclusion, put $y= T^n x$.
If $y\equiv 1 \pmod{3}$, 
then let $z\in \U_0$ be such that $T^{n}z = S y$.
Hence, $z\notin \bx_n$.
On the other hand, 
$T^{n+1}x = Ty = T(S y) = T^{n+1}z$,
which implies $z\in \bx_{n+1}$.
If $y\equiv 2 \pmod{3}$, 
then let $z\in \U$ be such that $T^{n}z = S^2 y$.
As above, we derive $z\in \bx_{n+1}\setminus \bx_n$.

Ad \ref{lemma:A:4}. From Corollary \ref{cor:Sj} it follows that,
for all $n\ge 1$, we have the identity $T^n = T^n \circ S^k$ on $\U$
and, hence, also on $\U_0$.
This implies $x\sim_n S^k x$, 
for all those $k\ge 0$ where $S^k x \in \U_0$.

Ad \ref{lemma:A:5}. Trivial.

Ad \ref{lemma:A:6}. 
Let $z\in T^{-1}\bra{Tx}_n$.
Then $Tz\in \bra{Tx}_n$,
which implies
$T^n(Tz)=T^{n+1}z = T^n(Tx)=T^{n+1}x$.
Hence, $z\in \bra{x}_{n+1}$.
This yields $T^{-1}\bra{Tx}_n \subseteq \bra{x}_{n+1}$.
For the converse, 
if $z\in \bra{x}_{n+1}$,
then $T^{n+1}z= T^n(Tz)= T^{n+1}x = T^n(Tx)$,
which implies $Tz\in \bra{Tx}_n$.
As a consequence, $z\in T^{-1}\bra{Tx}_n$.
We derive $\bra{x}_{n+1}\subseteq T^{-1}\bra{Tx}_n$.
\end{proof}

\begin{cor}\label{cor:images}
For all $x\in \U_0$, and for all $n\ge 0$,
\[
T \bx_{n+1} = \bra{Tx}_n, \quad 
	T^{-1}\bra{x}_n = \bra{\tau(x)}_{n+1}, \quad
	T\bra{\tau(x)}_{n+1}= \bra{x}_n.
\]
This is due to the surjectivity of $T$, 
see identity (\ref{eqn:invertingT}).
In addition, $T\bx_0=\bra{Tx}_0$.
\end{cor}

\begin{cor}
For all $n\ge 0$,
we may partition $\U_0$ as follows:
\begin{equation}\label{eqn:partition}
\U_0 = \bigcup_{x\in \U_0} \bra{x}_n.
\end{equation}
\end{cor}

Due to the strict inclusion $\bra{x}_n \subset \bra{x}_{n+1}$,
if we pass from $n$ to $n+1$,
this will result in a `reduction'  in the number of different equivalence classes.
Hence,
if $n$ increases, we get less and less elements in the partitions
(\ref{eqn:partition}) of $\U_0$.
As we will see in Corollary \ref{cor:equivalentresults},
the 3x+1 conjecture is equivalent to a collapse of this 
 sequence of nested partitions to a trivial partition of $\U_0$ consisting
of a single set.

\begin{cor}
For all $x\in \U_0$,
the limit $\lim_{n\to \infty} \delta_n(x)$ exists.
This is due to the fact that the sequence of positive integers $(\delta_n(x))_{n\ge 0}$ is
decreasing and bounded from below by $1$, hence convergent.
\end{cor}

We observe that the $3x+1$ conjecture is equivalent to  
$\lim_{n\to \infty} \delta_n(x) =1$ for all $x$ in $\U_0$.
It is also equivalent to $1\in \bigcup_{n\ge 0} \bx_n$,
for all $x$ in $\U_0$.

In the next step, we determine $\lim_{n\to \infty} \delta_n(x)$ and characterize
the union of the sets $\bx_n$, $n\ge 0$.

\begin{defn}
For $x,y\in \U_0$, we define the relation ``$\sim_\infty$'' on $\U_0$ as
\[
x\sim_\infty y \Leftrightarrow \exists n\in \N_0: \quad T^n x = T^n y.
\]
Further, put
$\bx_\infty = \left\{ z\in \U_0: \exists n\in \N_0 \text{ such that } T^n z= T^n x  \right\}$,
and  $\delta_\infty(x) = \min \bx_\infty$.
\end{defn}

\begin{lem}\label{lemma:B}
The following holds.
\begin{enumerate}
\item\label{lemma:B:1} The relation `$\sim_\infty$' is an equivalence relation on $\U_0$
and, for all $x\in \U_0$,
the set $\bx_\infty$ is the equivalence class 
of $x$ with respect to this equivalence relation.
Further, $\bx_\infty = \bra{\delta_\infty(x)}_\infty$,
and
	\[
	\bx_\infty = \bigcup_{n\ge 0} \bx_n.
	\]
\item\label{lemma:B:4} For all $x\in \U_0$,
	\[
	\delta_\infty(x) = \lim_{n\to \infty}  \delta_n(x).
	\]
\end{enumerate}
\end{lem}

\begin{proof}
Ad \ref{lemma:B:1}. 
This is easily verified.

Ad \ref{lemma:B:4}. 
Let $\delta'(x) =\lim_{n\to \infty} \delta_n(x)$.
Due to the fact that we are dealing with a convergent integer sequence,
there exists an integer $N$ such that for all $n\ge N$, $\delta'(x)=\delta_n(x)$.
From the fact $\delta_n(x)\in \bra{x}_n$,
 it follows that $\delta'(x)\in [x]_\infty$.
Hence, $\delta_\infty(x) \le \delta'(x)$.

On the other hand, for any $z\in \bra{x}_\infty$, there exists $n\in \N_0$ such that
$z\in [x]_n$. This implies that $z\ge \delta_n(x) \ge \delta'(x)$.
We note that, 
by definition, 
$\delta_\infty(x)\in \bra{x}_\infty$.
Consequently, $\delta_\infty(x) \ge \delta'(x)$.
\end{proof}

\begin{cor}\label{cor:equivalentresults}
The sets $\bra{x}_\infty$, $x\in \U_0$, form a partition of $\U_0$,
$
\U_0 = \bigcup_{x\in \U_0} \bra{x}_\infty.
$
The $3x+1$conjecture is equivalent to  each of the following statements:
\begin{enumerate}
\item $\U_0=\bra{1}_\infty$.
\item $\forall x\in \U_0: \quad \delta_\infty(x)=1$.
\end{enumerate}
\end{cor}

Let us study the action of $T$ on the sets $\bra{x}_\infty$.

\begin{lem}\label{lem:imageOfxinf}
For all $x\in \U_0$,
\[
T^{-1}\bra{Tx}_\infty = \bra{x}_\infty.
\]

\end{lem}

\begin{proof}
We have the following chain of equivalences:
\begin{align*}
z\in T^{-1}\bra{Tx}_\infty &\Leftrightarrow Tz\in \bra{Tx}_\infty
	\Leftrightarrow \exists n\ge 0: Tz\in \bra{Tx}_n\\
	&\Leftrightarrow \exists n\ge 0: T^{n+1}z=T^{n+1}x\\
	&\Leftrightarrow \exists n\ge 0: z\in \bra{x}_{n+1}
	\Leftrightarrow  z\in \bra{x}_\infty.
\end{align*}
\end{proof}

\begin{cor}\label{cor:furtherimages}
In analogy to Corollary \ref{cor:images}, 
for all $x\in \U_0$,
we have
\[
T \bx_\infty = \bra{Tx}_\infty, \quad 
	T^{-1}\bra{x}_\infty = \bra{\tau(x)}_\infty, \quad
	T\bra{\tau(x)}_\infty= \bra{x}_\infty.
\]
\end{cor}

\begin{rem}
For the analog equivalence classes on the set $\U$ instead of $\U_0$,
one has  $T\bra{x}_\infty \subset \bra{Tx}_\infty$, 
i.e., \emph{strict} inclusion,
as the inverse image $T^{-1}\{z\}$ of a point $z\in \bra{Tx}_\infty$  may be empty.
\end{rem}

In view of Lemma \ref{lem:imageOfxinf} and Corollary \ref{cor:furtherimages}
we may introduce the following map.

\begin{defn}\label{defn:inducedtransformation}
Let $\Uinf$ denote the quotient space $\U_0/\sim_\infty$, i.e., 
the set of equivalence classes associated with the equivalence relation `$\sim_\infty$'.
The induced map $\Tstar$ on  $\Uinf$ is defined as
\[
\Tstar \brainf{x} = \brainf{Tx}, \quad x\in \U_0.
\]
\end{defn}

\begin{thm}\label{thm:propertiesOfT}
The map $\Tstar$ on $\Uinf$ has the following properties:
\begin{enumerate}
\item \label{lem:propertiesOfT:0}$\Tstar$ is well defined.
\item \label{lem:propertiesOfT:1} $\Tstar$ is a bijection on $\Uinf$.
\item \label{lem:propertiesOfT:2} $\Tstar$ has the unique fixed point $\brainf{1}$: 
\[
\brainf{x}=\Tstar \brainf{x} \Rightarrow \brainf{x}=\brainf{1}.
\] 
\item \label{lem:propertiesOfT:3} It $x$ is a periodic point of $T$,
$T^kx=x$, with $k\ge 1$ the minimum period of $x$, then 
$\brainf{x}$ is a periodic point of $\Uinf$ with $T^{*k} \brainf{x}=\brainf{x}$.
\end{enumerate}
\end{thm}

\begin{proof}
Ad  \ref{lem:propertiesOfT:0}. 
We have to show that the value of $\Tstar$ is independent of the representative of the
equivalence class $\bra{x}_\infty$.
Suppose that $\bra{x}_\infty = \bra{z}_\infty$.
Then
there exists $n\ge 0$ such that $T^n x = T^n z$.
Hence, $T^{n+1} x = T^{n+1} z$, which implies
$Tz\in \bra{Tx}_n\subset \brainf{Tx}$.
Thus,  $\bra{Tx}_\infty = \bra{Tz}_\infty$.

Ad \ref{lem:propertiesOfT:1}. 
By Lemma \ref{lem:imageOfxinf}, $\Tstar \brainf{x} = \Tstar \brainf{z}$ implies $\brainf{Tx}=\brainf{Tz}$.
Hence, there exists $n\ge 0$ such that $T^n (Tx) = T^n (Tz)$.
As a consequence, $z\in \bra{x}_{n+1}\subset \brainf{x}$,
which implies $\brainf{z}=\brainf{x}$.
Thus, $\Tstar$ is injective on $\Uinf$.

Let $\brainf{y}$ be an arbitrary element of $\Uinf$.
Then $\brainf{y}= \brainf{T\tau(y)} =  \Tstar \brainf{\tau(y)}$.
Thus, $\Tstar$ is surjective.

Ad \ref{lem:propertiesOfT:2}. 
Suppose that $\brainf{x}=\Tstar \brainf{x}= \brainf{Tx}$.
Then there exists $n\ge 0$ such that $T^n x = T^n(Tx) = T(T^n x)$.
As a consequence, the element $T^n x$ is a fixed point in $\U_0$ under $T$.
This implies $T^n x = 1$ (see Remark \ref{rem:rem1}, Part 3.).
Due to $T^n 1 = 1$, we have $\brainf{x}=\brainf{1}$.

Ad \ref{lem:propertiesOfT:3}. The property $T^k x = x$ implies
$\brainf{T^k x} = T^{*k} \brainf{x} = \brainf{x}$.
\end{proof}



\section{Appendix}
The idea underlying our approach to the $3x+1$ conjecture was to find a
suitable \emph{metric space} $X$ in the form of some quotient space 
$X=\{ \overline{x}: x\in \U_0\}$,
and a \emph{contraction} $T^*$ on $X$ that is intrinsically related to 
the map $T$ in the sense that convergence of the sequences $(T^{*k}\overline{x})_{k\ge 0}$,
$\overline{x}\in X$, to the unique
fixed point of $T^*$ implies the convergence of the sequences $(T^kx)_{k\ge 0}$ to $1$,
i.e., the validity of the  $3x+1$ conjecture.

We were unable to realize this `dream' of applying the Banach fixed-point theorem,
because we have not found an appropriate pair $(X,T^*)$.
For example, $X=\U_{0, \infty}$ can easily be made into a metric space but it is the
proof of the contraction property of the induced map $T^*$ with respect to the chosen
metric where we failed.

The following concepts allow a somewhat deeper understanding of 
the dynamics of the map $T$.

\begin{defn}
The \emph{invertible} accelerated Collatz function
$f: \U_0 \rightarrow \U_0$ is defined as follows.
For $x\in \U_0$ and for $k\ge 0$, 
define $f^kx=T^k x$.
For $k<0$, put $f^kx = \tau^{-k}(x)$.
\end{defn}

Consider the following equivalence relation on $\U_0$:
\[
x\sim y \Leftrightarrow \exists m,n\ge 0: \ T^mx=T^ny.
\]
We write $\bx$ for the equivalence class of $x\in \U_0$, 
and get the following.

\begin{lem}\label{lem:relation2}
Let the relation $`\sim'$
be defined as above
and write $\Ustrich$ for the quotient space $\U_0/\sim$.
Then
\begin{enumerate}
\item For all $x\in \U_0$, the sets $\bx$ are $T$-invariant in the following sense:
\[
T^{-1}\bx = \bx, \quad T\bx = \bx = \bTx.
\]

\item For all $x\in \U_0$,
	\[
	\bx = \bigcup_{k\in \Z} \bra{f^kx}_\infty 
	= \cdots \cup \bra{\tau(x)}_\infty \cup \bx\cup\bTx\cup \cdots
	\]
\item We have $\bra{1} = \bra{1}_\infty$.
\item The map $T': \Ustrich \rightarrow \Ustrich$, $T'\bx=\bTx$ is well-defined
and every element $\bx$ of $\Ustrich$ is a fixed point of $T'$.
\end{enumerate}
\end{lem}

\begin{proof}
The proof is straightforward and employs the techniques introduced in Section \ref{s:preliminaries}.
\end{proof}

\begin{rem}
The reader should note the behavior of the class $\bra{1}$,
which is remarkably different from all other classes.
\end{rem}

\begin{rem}
For $x\in \U_0$, 
put $\delta(x) = \min \bx$.
The $3x+1$ conjecture is equivalent $\U_0=\bra{1}$.
Further, it is equivalent to $\delta(x)=1$ for all $x\in \U_0$.
\end{rem}

\begin{rem}
Lemma \ref{lem:relation2} tells us that every set $\bx$ is $T$-invariant,
which is to say that $T^{-1}\bx = \bx$.
In addition to this result,
Theorem \ref{thm:propertiesOfT} shows that $\bra{1}=\bra{1}_\infty$ is the only
$T$-invariant set of the form $\bx_\infty$.
These two results call out for the application of concepts from the theory
of  dynamical systems, for example from ergodic theory.
Let $(X, \mathcal{B}, m)$ be a probability space.
A measure preserving map $f: X\rightarrow X$ is called ergodic
if the only $f$-invariant elements $A$ of $\mathcal{B}$,
i.e., $f^{-1} A = A$,
are those with $m(A)=0$ or $m(A)=1$.
It is well known that ergodicity of $f$ is equivalent to each of the following
properties: (i) for every $A\in \mathcal{B}$ with $m(A)> 0$ we have
$m(\cup_{k=1}^\infty f^{-k}A)=1$, or 
(ii) for every $A,B\in \mathcal{B}$ with $m(A)>0$ and $m(B)>0$,
there exists $k>0$ with
$m(f^{-k}A\cap B) > 0$
(see, for example, Walters\cite[Theorem 1.5]{Walters82a}).
In our case, we would have to prove ergodicity for $f=T$,
where $\U_0$ would have to be equipped with an appropriate probability space structure.
We would then be able to derive the $3x+1$ conjecture for almost all $x$.
\end{rem}

The following two notions allow some kind of ``bookkeeping'' 
when we iterate the map
$f$.
With every  $x\in \U_0$, we may associate two infinite matrices 
as follows.

\begin{defn}
Let $x\in \U_0$. 
We define the matrix of equivalence classes associated with $x$ as
$C(x)  = \left(  c_{k,n} \right)_{k\in \Z, n\ge 0}$,
where $c_{k,n} = \bra{f^k x}_n$.

In addition, we define the matrix of minimal elements associated with $x$ as
$M(x) = \left( \mu_{k,n}\right)_{k\in \Z, n\ge 0}$,
with $\mu_{k,n} = \min c_{k,n}$.
Let $\delta^*(x)$ denote the minimal element of the matrix $M(x)$.
\end{defn}

Clearly, we have $\mu_{k,n} = \delta_n(f^k x)$,
and $\delta^*(x)=\delta(x)$.
Note that if we fix the row index $k$,
then the row $(\mu_{k,n})_{n\ge 0}$ in $M(x)$ has a constant tail eventually, 
because the convergent sequence $(\delta_n(f^k x))_{n\ge 0}$ is constant
from  some index $N=N(k)$ onwards,
with every element then being equal to $\delta_\infty(f^k x)$.

There is even further `tail'-structure in $M(x)$: 
suppose that $\delta^*(x)$ is equal to $\mu_{k,n}$, where $k$ and $n$ are minimal
with this property (in this order).
Then $\delta^*(x)=\mu_{k,m}$ for all $m\ge n$. That is to say,
the $k$-th row becomes eventually constant.

From the discussion above it follows that it is sufficient to prove
the $3x+1$ conjecture for the subset
$\{\delta_\infty(x): x\in \U_0  \}$
of $\U_0$ or, alternatively,
$\{\delta_1(x): x\in \U_0  \}$.
These facts suggest the following notion.

\begin{defn}
A subset $V$ of $\U_0$ is called sufficient
if the validity of the $3x+1$ conjecture for every element of $V$
implies the validity of the $3x+1$ conjecture for every element of $\U_0$.
\end{defn}

\begin{lem}
The set $\left\{ x\in \U_0: 1\le \nu_2(3x+1)\le 4 \right\}$ is sufficient.
\end{lem}

\begin{proof}
Let $x\in \U_0$ be arbitrary.
Trivially, we have $\bx_\infty = \bra{\delta_1(x)}_\infty$.
As a consequence, 
the set $\{\delta_1(x): x\in \U_0  \}$ is sufficient.
Further, $\delta_1(x)=\tau(Tx)$.

From Lemma \ref{lem:inverseofy} it follows that
$\nu_2(3\xi(x)+1)\in \{1,2\}$, 
for all $x\in \U_0$.
Due to the fact that either $\tau(x)=\xi(x)$, or $\tau(x)=S\xi(x)$,
we have $\nu_2(3\tau(x)+1)\in \{1,2,3,4\}$.
\end{proof}

The reader might want to compare this result with Sander \cite[Theorem 1]{Sander90a}.
For further, very extensive results on sufficient sets
we refer the reader to Monks\cite{Monks13a}.

\subsection*{Acknowledgements}
The author would like to thank Harry (Hillel) Furstenberg, who brought this problem to his attention
in several personal discussions some decades ago at CIRM in Luminy, France.

\def\cdprime{$''$} \def\cdprime{$''$}


\begin{small}
\noindent\textbf{Author's address:}\\
Peter Hellekalek,
Dept. of Mathematics, University of Salzburg, Hellbrunnerstrasse 34, 5020 Salzburg, Austria\\
E-mail: \texttt{peter.hellekalek@sbg.ac.at}
\end{small}

\end{document}